\documentclass[reqno]{amsart}
\usepackage{mathrsfs}
\newtheorem{theorem}{Theorem}
\newtheorem{thm}{Theorem}

\newtheorem{lemma}{Lemma}
\begin{document}

\title[Schr\"{o}dinger operators on star graphs]{Schr\"{o}dinger operators on star graphs with singularly scaled potentials supported near the vertices}%
\author{S. S. Man'ko}%
\address{Institute for Applied Problems of Mechanics and Mathematics, 3b Naukova st., 79601 L'viv, Ukraine}%
\email{stepan.manko@gmail.com}%


\begin{abstract}
We study Schr\"{o}dinger operators on star metric graphs with potentials of the form $\alpha\varepsilon^{-2}Q(\varepsilon^{-1}x)$.
In dimension 1 such potentials, with additional assumptions on $Q$, approximate in the sense of distributions as $\varepsilon\to0$ the
first derivative of the Dirac delta-function.
We establish the convergence of the Schr\"{o}dinger operators in the uniform resolvent topology and show that the limit operator depends on $\alpha$ and $Q$ in a very nontrivial way.
\end{abstract}
\maketitle
\section{Introduction}
Currently, there is considerable interest in theory of differential operators on graphs.
The reason for this lies in a great deal of progress in fabricating graph-shaped structures of submicron sizes, for which differential operators on graphs represent a natural model.
A case of special interest arises when operators appearing in such models have coefficients supported on graph vertices.
Such singular operators are widely used in
various application to atomic, nuclear, and solid state physics  (see the survey \cite{Kuchment02} for details).
By a quantum graph we mean a metric graph together with a second
order self-adjoint differential operator, which is determined
by differential operators on the edges and certain interface conditions at the vertices.
A quantum graph is a natural generalization of the
one-dimensional Schr\"{o}dinger operator; it can be applied to describe the transition of a quantum particle along the graph and
serves as a model of wave propagation in ``thin'' media.
The differential operators determine the motion of the quantum particle along the edges, while
the interface conditions describe the movement across the vertices.
The interface conditions have to be chosen to make the Hamiltonian
self-adjoint or, in physical terms, to ensure the conservation of the probability current at the vertex.
Basic notions and many
references on this subject can be found in \cite{KostrykinSchrader99, KostrykinSchrader00, NaimarkSolomyak00, Solomyak03, Solomyak04}.

The idea to investigate quantum mechanics of particles confined to a graph originated
with the study of free-electron models of organic molecules \cite{Pauling36, Platt49, RichardsonBalazs72}.
Quantum graphs have recently found numerous applications in physics, chemistry, engineering and quantum computing.
Among the systems that
were successfully modeled by quantum graphs, we mention, e.g. single-mode acoustic and
electromagnetic waveguide networks \cite{FlesiaJohnstonKunz87}, Anderson transition \cite{Anderson81} and quantum Hall systems
with long-range potential \cite{ChalkerCoddington88}, fracton excitations in fractal structures \cite{AvishaiLuck92}, and mesoscopic
quantum systems \cite{KowalSivanEntinWohlmanImry90}.
Applications also arise in quantum wires, quantum chaos and
photonic crystals (see \cite{KostrykinSchrader03, KottosSmilansky97, KuchmentKunyansky02}).
For surveys of physical systems giving rise to boundary value problems on graphs
see \cite{Kuchment02, Kuchment04} and the references
given there.

Many important one-dimensional models have their analogues on graphs.
One of the best studied graph Laplacians is given by the Schr\"{odinger} differential expressions on the graph edges and the $\delta$ coupling at $N$ edge vertices
\begin{equation*}
    \psi_1(a)=\ldots=\psi_N(a),\qquad \sum\limits_{n=1}^N\psi'_n(a)
    =\alpha\psi(a)
\end{equation*}
(see, e.g., \cite{Exn1, Exn2, Exn3, Exn4}).
Such a model is a generalization of the Schr\"{o}dinger operator with the potential $\alpha\delta(x)$.
The waveguide approximation to the $\delta$ coupling was proposed in \cite{ExnerPost09}.

Our aim in this paper is to study a family of Schr\"{o}dinger operators of the form
\begin{equation}\label{Regularization}
-\frac{\mathrm{d}^2}{\mathrm{d} x^2}
+\alpha\varepsilon^{-2}Q(\varepsilon^{-1}x)
\end{equation}
on a star metric graph.
Here $\varepsilon$ is a small positive parameter, $\alpha\in\mathbb{R}$, and $Q$ is a real-valued integrable function.
We shall establish the convergence of this family as $\varepsilon\to0$ in the norm resolvent topology.
In this paper we generalize the results of \cite{GolHry}, where such a problem was considered on the line (the graph with two edges).
In dimension one such potentials are often referred to as $\delta'$-like potentials, since
if $Q$ has zero mean and its first moment is equal to $-1$, then the sequence $\varepsilon^{-2}Q(\varepsilon^{-1}x)$ converges in the sense of distributions as $\varepsilon\to0$ to $\delta'(x)$.
Here $\delta'$ is the first derivative of the Dirac delta-function.
The Schr\"{o}dinger operators with the Dirac delta-function and its derivatives in potentials have widely been used in quantum mechanics and mathematical physics; see \cite{Alb2ed, AlbKur} and the references given there in for various physical models leading to such potentials.

In 1986 \v{S}eba~\cite{Seb} considered a formal Schr\"{o}dinger
operator with the potential $\delta'(x)$; he approximated $\delta'$ by regular potentials $\varepsilon^{-2}Q(\varepsilon^{-1}\cdot)$ and then investigated the convergence of the corresponding family of regular Schr\"{o}dinger operators of the form~\eqref{Regularization}.
\v{S}eba claimed that the regularized operators converge in the
uniform resolvent sense to the direct sum of the unperturbed
half-line Schr\"{o}dinger operators subject to the Dirichlet boundary conditions at the origin.
From the viewpoint of the scattering theory it means that the $\delta'$-barrier is completely opaque, i.e.,  in the limit $\varepsilon\to0$ the potential  $\varepsilon^{-2}Q(\varepsilon^{-1}\cdot)$ becomes a totally reflecting wall at the origin splitting the system into two independent subsystems lying on the half-lines $\mathbb{R}_-$ and $\mathbb{R}_+$.

Until recently, it was thought that the $\delta'$-potential is physically trivial in the above sense.
However, in 2003 Zolotaryuk et al. \cite{ChrArnZolErmGai} observed resonances in the transmission probability for piecewise cons\-tant $\delta'$-like potentials,
which conflicted with conclusions reached by \v{S}eba.
Although in \cite{ChrArnZolErmGai} the total reflection for the limiting $\delta'$-potential was demonstrated to take place for almost all $\alpha\in\mathbb{R}$, the authors found a discrete set of $\alpha$ (the roots of a transcendent equation depending on $Q$), for which partial transmission
through the limiting $\delta'$-potential occurs.
Exactly solvable models with other piecewise constant potentials as well as nonrectangular
regularizations of $\delta'$  have later been studied in \cite{Zol1}--\cite{Zol5} and the same conclusion has been drawn.
All these findings were generalized in \cite{Man1}, where the scattering on an arbitrary potential of the form $\alpha\varepsilon^{-2}Q(\varepsilon^{-1}x)$ with $Q$ supported by $[-1,1]$ has been considered.
It was proved that such a potential is asymptotically transparent only if the constant $\alpha$ is such that the problem
\begin{equation}\label{0ResonantSet}
    -f''+\alpha Qf=0\quad\text{on}\quad(-1,1),\qquad f'(-1)=f'(1)=0
\end{equation}
admits a nontrivial solution.

In \cite{GolMan1, GolMan2} the authors studied Schr\"{o}dinger operators of the form
\begin{equation*}
    -\frac{\mathrm{d}^2}{\mathrm{d} x^2} +\alpha\varepsilon^{-2} Q(\varepsilon^{-1}x)+ q(x)
\end{equation*}
with $q$ being a real-valued potential tending to $+\infty$ as $|x|\to\infty$; such a behavior ensures that the spectrum of the considered operators is discrete.
For each pair $(\alpha,Q)$ a limit selfadjoint operator was constructed there.
The choice of the limit operator was determined by proximity of its eigenvalues and eigenfunctions to those for the Schr\"{o}dinger operators with regularized potentials for small $\varepsilon$.
It was established that for a fixed $Q$ and almost all constants $\alpha\in\mathbb{R}$ the limit operator is just the direct sum of the Schr\"{o}dinger operators with potential $q$ on half-axes subject to the Dirichlet boundary conditions at the origin.
But in the exceptional case, when problem~\eqref{0ResonantSet} admits a nontrivial solution $f_\alpha$,
functions $\psi$ in the domain of the limit operators satisfy the interface conditions
\begin{equation*}
    f_\alpha(-1)\psi(+0)=f_\alpha(1)\psi(-0),\qquad f_\alpha(1)\psi'(+0)=f_\alpha(-1)\psi'(-0).
\end{equation*}

Studies of \cite{GolMan1, GolMan2} were continued in \cite{GolHry}, where the question on correct definition of the formal Schr\"{o}\-din\-ger operator with the potential $\delta'$ was finally answered.
The authors not only pointed out a mistake in \v{S}eba's proof, but also showed that the operators considered in \cite{GolMan1, GolMan2} converge in the
uniform resolvent sense as $\varepsilon\to0$ to the limit one obtained there.
The results of \cite{GolHry} were derived for the special case $q=0$; but the same can be obtained without difficulty for $q\neq0$ as well.
In \cite{AlbCacFin07, CacExn07} the problem of approximating a smooth quantum waveguide with a quantum graph was analyzed;
the authors encountered the question on the
convergence of a family of one-dimensional Schr\"{o}dinger operators of the form \eqref{Regularization} and obtained similar results
under the assumption that the mean value of $Q$ is different from zero.
For a treatment of a more general case of the $\alpha\delta'+\beta\delta$-like potential we refer the reader to the recent papers \cite{Gol1, Gol2}.

In the following section we briefly sketch the findings of \cite{Man2, Man3}, where an analogue for the one dimensional Schr\"{o}dinger operator with the $\delta'$-like potential was considered on the metric graph.
In Section 3 we prove our main theorem, which generalizes the results of \cite{GolHry} to the metric graph.

\section{Preliminaries and main results}
Let us recall the basic notions of the theory of differential equations on graphs.
By a metric graph $G=(V,E)$ we mean a finite set $V$ of points in $\mathbb{R}^3$ (\textit{vertices}) together with a set $E$ of smooth regular curves connecting the vertices (\textit{edges}).
A map $f:\,G\to\mathbb{R}$ is said to be a \textit{function on the graph}, and
the restriction of $f$ onto the edge $g\in E(G)$ will be denoted by $f_g$.
Each edge is equipped with a natural parametrization;
the differentiation is performed with respect to the natural parameter.
We denote by $\frac{\mathrm{d} f}{\mathrm{d} g}(a)$ the limit value of the derivative at the point $a\in V(G)$ taken in the direction away from the vertex.
The integral of $f$ over $G$ is the sum of integrals over all edges.
Let $\mathscr{L}_2(G)$ be a Hilbert space with the inner product $(f,h)=\int_Gf\bar{h}\,\mathrm{d} G$.
We also introduce the spaces
\begin{align*}
\mathscr{C}^\infty(G)&=\{f \mid
f_g\in C^\infty(\bar g)\;\hbox{for all}\; g\in E(G)\},
\\
\mathscr{AC}(G)&=\{f \mid f_g\in AC(g)\;
\hbox{for all}\; g\in E(G)\},
\\
\mathscr{W}_2^2(G)&=\{f\in \mathscr{L}_2(G) \mid f_g\in W_2^2(g)\;
\hbox{for all}\; g\in E(G)\}.
\end{align*}
We say that a function $f$ satisfies the Kirchhoff interface conditions at the vertex $a\in V(G)$ if $f$ is continuous at this vertex and $\sum\frac{\mathrm{d} f}{\mathrm{d} g}(a)=0$, where the sum is taken over all $g\in E(G)$ such that $a\in\bar g$;
in the latter case we shall write $f\in\mathcal{K}(G;a)$.

Let us consider a noncompact star graph $\Gamma$ consisting of three edges $\gamma_1$, $\gamma_2$ and $\gamma_3$.
All edges are connected at the origin of $\mathbb{R}^3$, denoted by $a$.
Then $E(\Gamma)=\{\gamma_1,\gamma_2,\gamma_3\}$ and suppose that all edges are half-lines.
We write $a_n^\varepsilon$ for the point of intersection of the $\varepsilon$-sphere, centered at $a$, with the edge $\gamma_n\in E(\Gamma)$;
denote by $\Gamma_\varepsilon$ a sub-partition of $\Gamma$ containing new vertices $a_{1}^\varepsilon$, $a_2^\varepsilon$ and $a_3^\varepsilon$.
Each $a_n^\varepsilon$ divides the edge $\gamma_n$ of the graph $\Gamma$ into two edges $\omega_n^\varepsilon$ (compact) and $\gamma_n^\varepsilon$ (noncompact) of the graph $\Gamma_\varepsilon$.
Let $\Omega_\varepsilon$ be a star subgraph of $\Gamma_\varepsilon$ such that
$V(\Omega_\varepsilon)=\{a,a_1^\varepsilon,a_2^\varepsilon,a_3^\varepsilon\}$ and
$E(\Omega_\varepsilon)=\{\omega_1^\varepsilon,\omega_2^\varepsilon,\omega_3^\varepsilon\}$.
Set $\Omega:=\Omega_1$, $\omega_n:=\omega_n^1$ and $a_n:=a_n^1$.

In \cite{Man2}, the family of Schr\"{o}dinger operators on the star graph
\begin{align*}
\mathrm{H}_\varepsilon(\alpha,Q)&=
-\frac{\mathrm{d}^2}{\mathrm{d} x^2}
+q(x)+\alpha\varepsilon^{-2}Q(\varepsilon^{-1}x),
\\
\mathrm{dom}\,\mathrm{H}_\varepsilon(\alpha,Q)&=
\big\{f\in \mathscr{L}_2(\Gamma)\mid
f,\,f'\in \mathscr{AC}(\Gamma),\quad -f''+qf\in \mathscr{L}_2(\Gamma),\quad f\in \mathcal{K}(\Gamma;a)\big\}
\end{align*}
was studied.
Here $Q\in\mathcal{Q}:=\big\{Q\in \mathscr{C}^\infty(\Gamma)\mid \mathrm{supp}\,Q=\Omega,\;\int_\Omega Q\,\mathrm{d}\Omega=0\big\}$
and the potential $\varepsilon^{-2}Q(\varepsilon^{-1}x)$ is referred to as the $\delta'$-\textit{like potential};
$q$ is a smooth real-valued function such that
$q_\gamma(x)\to+\infty$ as $|x|\to+\infty$ for all $\gamma\in E(\Gamma)$ (this ensures the spectrum discreteness for $\mathrm{H}_\varepsilon(\alpha, Q)$).
An operator $\mathrm{H}(\alpha,Q)$ was assigned to each pair $(\alpha,Q)\in\mathbb{R}\times\mathcal{Q}$.
The choice of the operator was suggested by the proximity of the eigenvalues and eigenfunctions for $\mathrm{H}_\varepsilon(\alpha,Q)$ and $\mathrm{H}(\alpha,Q)$ respectively.
The following problem  is introduced in \cite{Man2}
\begin{align}\label{Resonant_Set}
\begin{aligned}
-&f''+\alpha Qf=0\quad\hbox{on$\quad\Omega$,}
\qquad
f\in\mathcal{K}(\Omega;a),
\\
&\frac{\mathrm{d} f}{\mathrm{d}\omega_1}(a_1)=
\frac{\mathrm{d} f}{\mathrm{d}\omega_2}(a_2)=
\frac{\mathrm{d} f}{\mathrm{d}\omega_3}(a_3)
=0.
\end{aligned}
\end{align}
The choice of $\mathrm{H}(\alpha,Q)$ depends on whether the above problem has a nontrivial solution which can be regarded as an eigenfunction corresponding to the eigenvalue $\alpha$.
Therefore three different cases are distinguished.
In the simplest \textit{non-resonant case}, when $\alpha$ is not in the spectrum of problem \eqref{Resonant_Set}, $\mathrm{H}(\alpha,Q)$ is the direct sum of the Schr\"{o}dinger operators with the potential $q$ on edges, subject to the Dirichlet boundary conditions at the vertex $a$.
In the {\it(simple or double) resonant case} the \textit{coupling vector} $(\theta_1, \theta_2, \theta_3)$ is introduced.

Let $\alpha$ be a simple eigenvalue of the problem~(\ref{Resonant_Set}) (the simple resonant case) with an eigenfunction $\phi_\alpha$; then
introduce
\begin{equation*}
    \theta_1:=\phi_\alpha(a_1),\qquad
    \theta_2:=\phi_\alpha(a_2),\qquad
    \theta_3:=\phi_\alpha(a_3).
\end{equation*}
In the double resonant case, when $\alpha$ is a double eigenvalue,
the coupling vector is defined as the vector product of $(\varphi_\alpha(a_1),\varphi_\alpha(a_2),\varphi_\alpha(a_3))$ and $(\psi_\alpha(a_1),\psi_\alpha(a_2),\psi_\alpha(a_3))$, where $\varphi_\alpha$ and $\psi_\alpha$ form a basis in the corresponding eigenspace, i.e.,
\begin{eqnarray*}
\theta_1:=
\begin{vmatrix}
\varphi_\alpha(a_2)&\varphi_\alpha(a_3)\\
\psi_\alpha(a_2)&\psi_\alpha(a_3)
\end{vmatrix}
,\quad
\theta_2:=
\begin{vmatrix}
\varphi_\alpha(a_3)&\varphi_\alpha(a_1)\\
\psi_\alpha(a_3)&\psi_\alpha(a_1)
\end{vmatrix}
,\quad
\theta_3:=
\begin{vmatrix}
\varphi_\alpha(a_1)&\varphi_\alpha(a_2)\\
\psi_\alpha(a_1)&\psi_\alpha(a_2)
\end{vmatrix}.
\end{eqnarray*}
In the simple resonant case the operator $\mathrm{H}(\alpha,Q)$ acts via
$\mathrm{H}(\alpha,Q)f=-f''+qf$ on an appropriate set of functions obeying the interface conditions
\begin{equation*}
\frac{f_{\gamma_1}(a)}{\theta_1}=
\frac{f_{\gamma_2}(a)}{\theta_2}=
\frac{f_{\gamma_3}(a)}{\theta_3},\qquad
\sum\limits_{n=1}^3\theta_n\frac{\mathrm{d} f}{\mathrm{d}\gamma_n}(a)=0.
\end{equation*}
In the double resonant case
the interface conditions may be written as
\begin{equation*}
\frac{1}{\theta_1}\frac{\mathrm{d} f}{\mathrm{d}\gamma_1}(a)=
\frac{1}{\theta_2}\frac{\mathrm{d} f}{\mathrm{d}\gamma_2}(a)=
\frac{1}{\theta_3}\frac{\mathrm{d} f}{\mathrm{d}\gamma_3}(a),\qquad
\sum\limits_{n=1}^3\theta_nf_{\gamma_n}(a)=0.
\end{equation*}

If $\lambda$ is an eigenvalue of the operator $\mathrm{H}(\alpha,Q)$, we shall denote by $\mathrm{P}_\lambda$ the orthogonal projector onto the corresponding eigenspace.
Let $\mathrm{P}_\lambda(\varepsilon)$ stand for the orthogonal projector onto the finite dimensional space spanned by all eigenfunctions corresponding to those eigenvalues $\lambda_\varepsilon$ of $\mathrm{H}_\varepsilon(\alpha,Q)$ that $\lambda_\varepsilon\to\lambda$ as $\varepsilon\to0$.
Here $\lambda_\varepsilon=\lambda_{j,\varepsilon}$ continuously depends on $\varepsilon$.
The results of \cite{Man2} may be summarized in the following theorem.

\begin{thm}
All eigenvalues of $\mathrm{H}_\varepsilon(\alpha,Q)$ (except at most a finite number) are bounded as $\varepsilon\to0$.
Let $\lambda_\varepsilon$ be an eigenvalue of $\mathrm{H}_\varepsilon(\alpha,Q)$ bounded as $\varepsilon\to0$,
then $\lambda_\varepsilon$ has a finite limit $\lambda$ that is a point of the spectrum of $\mathrm{H}(\alpha,Q)$. Moreover, $\|\mathrm{P}_\lambda(\varepsilon)-\mathrm{P}_\lambda\|\to0$ as $\varepsilon\to0$.
Conversely, if $\lambda$ is an eigenvalue of $\mathrm{H}(\alpha,Q)$, then there exists an eigenvalue $\lambda_\varepsilon$ of $\mathrm{H}_\varepsilon(\alpha,Q)$ such that $\lambda_\varepsilon\to\lambda$ as $\varepsilon\to0$.
\end{thm}

From this point forward we assume that $q\equiv0$, i.e., the operator $\mathrm{H}(\alpha,Q)$ involves no potential and $\mathrm{H}_\varepsilon(\alpha,Q)$ is
\begin{equation*}
\mathrm{H}_\varepsilon(\alpha,Q)
=-\frac{\mathrm{d}^2}{\mathrm{d} x^2}+\alpha\varepsilon^{-2}Q(\varepsilon^{-1}x),
\qquad
\mathrm{dom}\,
\mathrm{H}_\varepsilon(\alpha,Q)=
\mathscr{W}_2^2(\Gamma)\cap\mathcal{K}(\Gamma;a).
\end{equation*}
The scattering properties of the Hamiltonians $\mathrm{H}_\varepsilon(\alpha,Q)$ on the graph $\Gamma$ with the finite-range potentials $\alpha\varepsilon^{-2}Q(\varepsilon^{-1}x)$ in the limit $\varepsilon\to0$ were studied in \cite{Man3}.
It was proved that the scattering coefficients depend on $\alpha$ and $Q$.
In the generic (non-resonant) case the barrier $\alpha\varepsilon^{-2}Q(\varepsilon^{-1}x)$ is asymptotically opaque as $\varepsilon\to0$.
An exception to this occurs when problem \eqref{Resonant_Set} admits a nontrivial solution.
In \cite{Man3} it was also shown that
the scattering amplitude for the operators $\mathrm{H}_\varepsilon(\alpha,Q)$ and $\mathrm{H}_0$ converges as $\varepsilon\to0$ to that for the limiting operator $\mathrm{H}(\alpha,Q)$ and $\mathrm{H}_0$.
Here the operator $\mathrm{H}_0$ (the Hamiltonian  of a free particle on $\Gamma$ \cite{Kuchment02})  acts via $\mathrm{H}_0f=-f''$ on its domain $\mathscr{W}_2^2(\Gamma)\cap\mathcal{K}(\Gamma;a)$.
The following theorem summarizes the main advances in \cite{Man3}.

\begin{thm}
The scattering matrix for the operators $\mathrm{H}_\varepsilon(\alpha,Q)$
and $\mathrm{H}_0$ converges as $\varepsilon\to0$ to the scattering matrix for $\mathrm{H}(\alpha,Q)$ and $\mathrm{H}_0$.
\end{thm}

We observe that the above results on proximity for small $\varepsilon$ of the eigenvalues, eigen\-func\-tions, and scattering quantities for the limiting and regularized operators do not shed any light on convergence as $\varepsilon\to0$ of the operators $\mathrm{H}_\varepsilon(\alpha,Q)$ in whatever topology.
Such a convergence, however, would imply convergence of many other characteristics of interest.
Our objective in this paper is to study this question and the main result is contained in the following theorem.
\begin{theorem}
As  $\varepsilon\to0$, the operator family $\mathrm{H}_\varepsilon(\alpha,Q)$ converges  in the norm resolvent sense to $\mathrm{H}(\alpha,Q)$.
Moreover,
for a fixed $\zeta\in\mathbb{C}\setminus\mathbb{R}$ there exists a constant $C$ such that
\[
\|(\mathrm{H}_\varepsilon(\alpha,Q)-\zeta)^{-1}-
(\mathrm{H}(\alpha,Q)-\zeta)^{-1}\|_{B(\mathscr{L}_2(\Gamma))}
\leq  C\varepsilon^\frac12
\]
for every $\varepsilon\in(0,1)$.
\end{theorem}

Throughout the paper, $C_n$ and $c_n$ will denote different constants independent of $f$ and $\varepsilon$, and $\|\cdot\|$ will denote the $\mathscr{L}_2(\Gamma)$-norm.

\section{Proof of Theorem 1}
The underlying idea is to construct a function $\tilde y_\varepsilon$ that is a good approximation to the functions $y_\varepsilon:=(\mathrm{H}_\varepsilon(\alpha,Q)-\zeta)^{-1}f$ and $y:=(\mathrm{H}(\alpha,Q)-\zeta)^{-1}f$, uniformly for $f$ in bounded subsets of $\mathscr{L}_2(\Gamma)$.
Here $\zeta\in\mathbb{C}\setminus\mathbb{R}$ is a fixed number.
More precisely, to establish Theorem 1, we shall rely on the fact that for every $f\in\mathscr{L}_2(\Gamma)$ and $\varepsilon>0$ there exists a function $\tilde y_\varepsilon$ with the property that
\begin{equation}\label{Ineq_Approx}
\|\tilde{y}_\varepsilon-y_\varepsilon\|\leq C_1\varepsilon^\frac12\|f\|,\qquad
\|\tilde{y}_\varepsilon-y\|\leq C_2\varepsilon^\frac12\|f\|.
\end{equation}
From this  we find that
\[
\|(\mathrm{H}_\varepsilon(\alpha,Q)-\zeta)^{-1}f-
(\mathrm{H}(\alpha,Q)-\zeta)^{-1}f\|
\leq \|\tilde{y}_\varepsilon-y\|+\|y_\varepsilon-\tilde{y}_\varepsilon\|\leq C\varepsilon^\frac12\|f\|.
\]
Our next aim is to construct such an approximation.
The proof proceeds differently in the three cases: non-resonant, simple and double resonant.

\subsection{Proof of Theorem 1 in the non-resonant case}
In this subsection we shall prove Theorem 1 in the non-resonant case, when $\alpha$ does not belong to the resonant set.
We shall construct an approximation $\tilde y_\varepsilon$.

Denote by $u_\varepsilon$ a solution of the problem
\begin{equation*}
\begin{cases}
-u''+\alpha Qu=\varepsilon f(\varepsilon\,\cdot)\quad\mathrm{on}\quad\Omega,\qquad
u\in \mathcal{K}(\Omega;a),\\

\phantom{-}\frac{\mathrm{d} u}{\mathrm{d}\omega_n}(a_n)=\frac{\mathrm{d} y}{\mathrm{d}\gamma_n^\varepsilon}(a_n^\varepsilon),\quad
n=1,2,3.
\end{cases}
\end{equation*}
It is clear that $u_\varepsilon$ belongs to the Sobolev space $\mathscr{W}^2_2(\Omega)$ for every positive $\varepsilon$.

\begin{lemma}\label{Lemma_Non_Resonant}
Suppose that $f\in \mathscr{L}_2(\Gamma)$ and $\varepsilon\in(0,1)$; then
\[
\|u_\varepsilon\|_{\mathscr{W}_2^2(\Omega)}\leq C\|f\|.
\]
\end{lemma}
\begin{proof}
Observe that $(\mathrm{H}(\alpha,Q)-\zeta)^{-1}$ is a bounded operator from $\mathscr{L}_2(\Gamma)$ to the domain of $\mathrm{H}(\alpha,Q)$ equipped with the graph norm.
The latter space is a subspace of $\mathscr{W}_2^2(\Gamma)$, hence,
\[
\|y\|_{\mathscr{W}_2^2(\Gamma)}\leq c_1\|f\|.
\]
Since by the Sobolev embedding theorem $\mathscr{W}_2^2(\Gamma)\subset\mathscr{C}^1(\Gamma)$, we have
\begin{equation}\label{Sobolev}
\|y\|_{\mathscr{C}^1(\Gamma)}\leq c_2\|f\|.
\end{equation}
The sulution $u_\varepsilon$ obeys the a priori estimate
\[
\|u_\varepsilon\|_{\mathscr{W}_2^2(\Omega)}\leq
c_3
\Big(\sum\limits_{n=1}^3
\Big|\frac{\mathrm{d} y}{\mathrm{d}\gamma_n^\varepsilon}(a_n^\varepsilon)\Big|+
\varepsilon\|f(\varepsilon\,\cdot)\|_{\mathscr{L}_2(\Omega)}
\Big);
\]
this estimate can be obtained, roughly speaking, by removing the nonhomogeneity from the boundary conditions of the problem for $u_\varepsilon$ to the right hand side of the equation and using the fact that the resolvent of the corresponding differential operator
is a bounded operator from $\mathscr{L}_2(\Omega)$ to the domain of the operator equipped with the graph norm.
Applying \eqref{Sobolev}, the a priori estimate and taking into account that
\begin{equation}\label{0Norm_f_Omega}
\|f(\varepsilon\,\cdot)\|_{\mathscr{L}_2(\Omega)}\leq \varepsilon^{-\frac12}\|f\|,
\end{equation}
we, therefore, easily derive the inequality
$\|u_\varepsilon\|_{\mathscr{W}_2^2(\Omega)}
\leq C\|f\|$, which completes the proof.
\end{proof}

We introduce the function $v_\varepsilon:=(1-\chi_\varepsilon)y+\varepsilon\chi_\varepsilon u_\varepsilon$, where $\chi_\varepsilon$ is the characteristic function of $\Omega_\varepsilon$, i.e., $v_\varepsilon=y$ on $\Gamma\setminus\Omega_\varepsilon$ and $v_\varepsilon=\varepsilon\, u_\varepsilon(\varepsilon^{-1}\cdot)$ on $\Omega_\varepsilon$.
The function $v_\varepsilon$ is almost the desired approximation; the only problem is that it  is discontinuous at the points $a_n^\varepsilon$.
However the jumps of $v_\varepsilon$ at these points are small.
Indeed,
\[
[v_\varepsilon]_{a_n^\varepsilon}=
y(a_n^\varepsilon)-\varepsilon\, u_\varepsilon(a_n),\qquad
[v'_\varepsilon]_{a_n^\varepsilon}=0,
\]
where $[h]_{a_n^\varepsilon}=h_{\gamma_n^\varepsilon}(a_n^\varepsilon)-h_{\omega_n^\varepsilon}(a_n^\varepsilon)$ and
$[h']_{a_n^\varepsilon}=
\frac{\mathrm{d} h}{\mathrm{d}\gamma_n^\varepsilon}(a_n^\varepsilon)+\frac{\mathrm{d} h}{\mathrm{d}\omega_n^\varepsilon}(a_n^\varepsilon)$
are jumps of a function
$h$ and its first derivative at $x=a_n^\varepsilon$.
In view of Lemma~1
and the relations
\begin{equation*}
|y(a_n^\varepsilon)|
\leq \int_a^{a_n^\varepsilon}|y'|\,\mathrm{d}\gamma_n\leq c_1 \varepsilon^\frac12\|y\|_
{\mathscr{W}_2^2(\Gamma)}\leq
c_2\varepsilon^\frac12\|f\|
\end{equation*}
we get
\begin{align}
\label{0W_Jamps}
|[v_\varepsilon]_{a_n^\varepsilon}|\leq& |y(a_n^\varepsilon)|+
\varepsilon\,\|u_\varepsilon\|_{\mathscr{W}_2^2(\Omega)}\leq c_3\varepsilon^\frac12\|f\|.
\end{align}

Denote by $\eta_{0,n}^\varepsilon$ functions on the graph $\Gamma$ that are smooth outside the point $x=a_n^\varepsilon$, have
supports $[a_n^\varepsilon,a_n^{\varepsilon+1}]$, and
$\eta_{0,n}^\varepsilon=1$ on $[a_n^\varepsilon,a_n^{\varepsilon+\frac12}]$ for $n=1,2,3$.
The function $\eta_{0,1}^\varepsilon$ can be constructed in the following way.
On the edge $\gamma_1$ we consider a smooth function $\eta_{0,1}$ with the properties that $\mathrm{supp}\,\eta_{0,1}=[a,a_1^1]$ and $\eta_{0,1}=1$ on $[a,a_1^{\frac12}]$; then $\eta_{0,1}^\varepsilon$ is just a translation of $\eta_{0,1}$  extended by zero to the whole graph $\Gamma$.
The functions $\eta_{0,2}^\varepsilon$ and $\eta_{0,3}^\varepsilon$ can be constructed in a similar manner.

We set
\[
w_\varepsilon(x):=\sum_{n=1}^3[v_\varepsilon]_{a_n^\varepsilon}\eta_{0,n}^\varepsilon(x).
\]
Taking into account \eqref{0W_Jamps}, we obtain
\begin{equation}\label{0Ineq_w_eps}
\max\limits_{x\in\Omega_{\varepsilon+1}\setminus\Omega_\varepsilon}
|w_\varepsilon^{(j)}(x)|\leq
c_4\varepsilon^\frac12\|f\|
\end{equation}
for $j=0,1,2$.
By construction,
\begin{equation*}
\tilde{y}_\varepsilon:=v_\varepsilon-w_\varepsilon=
\begin{cases}
y-w_\varepsilon&\text{on}\quad \Gamma\setminus\Omega_\varepsilon,\\
\varepsilon\, u_\varepsilon(\varepsilon^{-1}\cdot)&\text{on}\quad \Omega_\varepsilon
\end{cases}
\end{equation*}
is a $\mathscr{C}^1(\Gamma)$-function and belongs to the domain of $\mathrm{H}_\varepsilon(\alpha,Q)$.
Now we show that $\tilde{y}_\varepsilon$ is a desired approximation
for $y_\varepsilon=(\mathrm{H}_\varepsilon(\alpha,Q)-\zeta)^{-1}f$ and $y=(\mathrm{H}(\alpha,Q)-\zeta)^{-1}f$.

\begin{proof}[Proof of Theorem 1 in the non-resonant case.]
Rewrite $\tilde{y}_\varepsilon$ in the form
\[
\textstyle
\tilde{y}_\varepsilon(x)=(1-\chi_\varepsilon(x))y(x)+
\varepsilon u_\varepsilon(\varepsilon^{-1}x)
-w_\varepsilon(x),
\]
where $\chi_\varepsilon$ is the characteristic function of $\Omega_\varepsilon$, $u_\varepsilon$ and $w_\varepsilon$ are extended by zero to the whole graph $\Gamma$.
Recalling the definition of $y$, $u_\varepsilon$, and $w_\varepsilon$, we find that
\[
\textstyle
(\mathrm{H}_\varepsilon(\alpha,Q)-\zeta)\tilde{y}_\varepsilon(x)=
(-\frac{\mathrm{d}^2}{\mathrm{d} x^2}-\zeta)(y(x)-w_\varepsilon(x))=f(x)+w''_\varepsilon(x)
+\zeta w_\varepsilon(x)
\]
for $x\in\Gamma\setminus\Omega_\varepsilon$, and that
\begin{align*}
(\mathrm{H}_\varepsilon(\alpha,Q)-\zeta)\tilde{y}_\varepsilon(x)&=
\varepsilon\big\{-\frac{\mathrm{d}^2}{\mathrm{d} x^2}+\alpha\varepsilon^{-2}Q(\varepsilon^{-1}x)-\zeta\big\}
 u_\varepsilon(\varepsilon^{-1}x)\\
&=
\varepsilon^{-1}\{-u''_\varepsilon+\alpha Qu_\varepsilon\}-\zeta\tilde{y}_\varepsilon(x)
=f(x)-\varepsilon\,\zeta u_\varepsilon(\varepsilon^{-1}x)
\end{align*}
for $x\in\Omega_\varepsilon$.
Therefore
$(\mathrm{H}_\varepsilon(\alpha,Q)-\zeta)\tilde{y}_\varepsilon(x)=f+r_\varepsilon$, where
\[
r_\varepsilon(x)=
\begin{cases}
w''_\varepsilon(x)+\zeta w_\varepsilon(x)&\text{if}\quad x\in\Gamma\setminus\Omega_\varepsilon,\\
-\varepsilon\,\zeta u_\varepsilon(\frac{x}{\varepsilon})&\text{if}\quad x\in\Omega_\varepsilon.
\end{cases}
\]
From this we conclude that
$\tilde{y}_\varepsilon-y_\varepsilon=(\mathrm{H}_\varepsilon(\alpha,Q)-\zeta)^{-1}r_\varepsilon$, and thus
\[
\|\tilde{y}_\varepsilon-y_\varepsilon\|\leq
\|(\mathrm{H}_\varepsilon(\alpha,Q)-\zeta)^{-1}\|\,\|r_\varepsilon\|\leq
|\Im\zeta|^{-1}\,\|r_\varepsilon\|.
\]
We can now use Lemma~\ref{Lemma_Non_Resonant} and estimate \eqref{0Ineq_w_eps} to arrive at the relation
\begin{align*}
\|r_\varepsilon\|&\leq c_1\|w''_\varepsilon+\zeta w_\varepsilon\|_{\mathscr{L}_2(\Omega_{\varepsilon+1}\setminus\Omega_\varepsilon)}+
c_2\|u_\varepsilon(\varepsilon^{-1}\cdot)\|_{\mathscr{L}_2(\Omega_\varepsilon)}\\
&\leq
c_3\max\limits_{x\in\Omega_{\varepsilon+1}\setminus\Omega_\varepsilon}
(|w_\varepsilon|+|w''_\varepsilon|)+
c_4\varepsilon^\frac12\|u_\varepsilon\|_{\mathscr{L}_2(\Omega)}
\leq c_5\varepsilon^\frac12\|f\|.
\end{align*}
This proves the first inequality in \eqref{Ineq_Approx}.
Similarly,
\begin{align*}
\|\tilde{y}_\varepsilon-y\|=&\,
\|\varepsilon u_\varepsilon(\varepsilon^{-1}\cdot)
-w_\varepsilon-\chi_\varepsilon y\|
\leq
c_6\varepsilon^\frac32\|u_\varepsilon\|_{\mathscr{L}_2(\Omega)}
\\&
+c_7\max\limits_{x\in\Omega_{\varepsilon+1}\setminus\Omega_\varepsilon}|w_\varepsilon|
+c_{8}\|y\|_{\mathscr{C}(\Gamma)}\|\chi_\varepsilon\|\leq c_{9}\varepsilon^\frac12\|f\|
\end{align*}
as required.
\end{proof}

\subsection{Proof of Theorem 1 in the simple resonant case}
Our aim in this subsection is to prove Theorem 1 in the simple resonant case, when $\alpha$ is a simple eigenvalue of problem \eqref{Resonant_Set}.

Let $\phi_\alpha$ be an eigenfunction of problem \eqref{Resonant_Set} corresponding to the eigenvalue $\alpha$.
Since all $\phi_\alpha(a_n)$ cannot be zero, we assume without loss of generality that $\phi_\alpha(a_1)=1$.
Denote by $u_\varepsilon$ a solution of the problem
\begin{equation}\label{1Eq v eps}
\begin{cases}
-u''+\alpha Qu=\varepsilon f(\varepsilon\,\cdot)\quad\mathrm{on}\quad\Omega,\qquad
u\in \mathcal{K}(\Omega;a),\\
\phantom{-}\frac{\mathrm{d} u}{\mathrm{d}\omega_1}(a_1)=\kappa_\varepsilon
,\quad
\frac{\mathrm{d} u}{\mathrm{d}\omega_n}(a_n)=
\frac{\mathrm{d} y}{\mathrm{d}\gamma_n^\varepsilon}(a_n^\varepsilon)
,\quad
n=2,3,
\end{cases}
\end{equation}
obeying the condition $u_\varepsilon(a_1)=0$.
Such a solution exists and is unique if
\begin{equation}\label{1kappa_eps}
\kappa_\varepsilon=
-
\Big(
\theta_2\frac{\mathrm{d} y}{\mathrm{d}\gamma_2^\varepsilon}(a_2^\varepsilon)+
\theta_3\frac{\mathrm{d} y}{\mathrm{d}\gamma_3^\varepsilon}(a_3^\varepsilon)+
\varepsilon\int_\Omega \phi_\alpha(t)f(\varepsilon t)\,\mathrm{d}\Omega
\Big);
\end{equation}
this follows from the Fredholm alternative.
In order to obtain $\kappa_\varepsilon$, we multiplied Eq. \eqref{1Eq v eps}
by the eigenfunction $\phi_\alpha$ and integrated by parts, bearing in mind that $\theta_1:=\phi_\alpha(a_1)=1$.

\begin{lemma}\label{Lemma_Simple_Resonant}
For any $f\in \mathscr{L}_2(\Gamma)$ and $\varepsilon\in(0,1)$ the following inequalities hold:
\[
\Big|\kappa_\varepsilon-\frac{\mathrm{d} y}{\mathrm{d}\gamma_1}(a)\Big|
\leq
C_1\varepsilon^\frac12\|f\|,
\qquad
\|u_\varepsilon\|_{\mathscr{W}_2^2(\Omega)}\leq C_2\|f\|.
\]
\end{lemma}
\begin{proof}
As in the proof of Lemma~\ref{Lemma_Non_Resonant} we find that
\[
\|y\|_{\mathscr{W}_2^2(\Gamma)}\leq c_1\|f\|,\qquad
\|y\|_{\mathscr{C}^1(\Gamma)}\leq c_2\|f\|.
\]
Subtracting the relation $\frac{\mathrm{d} y}{\mathrm{d}\gamma_1}(a)=
-\theta_2\frac{\mathrm{d} y}{\mathrm{d}\gamma_2}(a)-
\theta_3\frac{\mathrm{d} y}{\mathrm{d}\gamma_3}(a)$
from \eqref{1kappa_eps} gives
\begin{align*}
\Big|\kappa_\varepsilon-\frac{\mathrm{d} y}{\mathrm{d}\gamma_1}(a)\Big|\leq&\,
|\theta_2|
\Big|\frac{\mathrm{d} y}{\mathrm{d}\gamma_2^\varepsilon}(a_2^\varepsilon)-\frac{\mathrm{d} y}{\mathrm{d}\gamma_2}(a)\Big|+
|\theta_3|
\Big|\frac{\mathrm{d} y}{\mathrm{d}\gamma_3^\varepsilon}(a_3^\varepsilon)-\frac{\mathrm{d} y}{\mathrm{d}\gamma_3}(a)\Big|\\
&+\varepsilon\,\|f(\varepsilon\,\cdot)\|_{L_2(\Omega)}\|\phi_\alpha\|_{L_2(\Omega)}\leq
C_1\varepsilon^\frac12\|f\|.
\end{align*}
Here we used \eqref{0Norm_f_Omega} and the following estimates
\begin{equation}\label{1Jump_1_Der}
\Big|\frac{\mathrm{d} y}{\mathrm{d}\gamma_n^\varepsilon}(a_n^\varepsilon)-\frac{\mathrm{d} y}{\mathrm{d}\gamma_n}(a)\Big|
\leq \int_a^{a_n^\varepsilon}|y''|\,\mathrm{d}\gamma_n\leq c_3 \varepsilon^\frac12\|y\|_
{\mathscr{W}_2^2(\Gamma)}\leq
c_4\varepsilon^\frac12\|f\|
\end{equation}
for $n=1,2,3$.

Observe that the restriction $u_{\varepsilon,\omega_1}$ of $u_{\varepsilon}$ onto $\omega_1$ is a solution of the Cauchy problem
\begin{equation*}
\begin{cases}
-u''+\alpha Qu=\varepsilon f(\varepsilon\,\cdot)\quad\mathrm{on}\quad\omega_1,\\
\phantom{-}u(a_1)=0,\quad\frac{\mathrm{d} u}{\mathrm{d}\omega_1}(a_1)=\kappa_\varepsilon;
\end{cases}
\end{equation*}
thus, using \eqref{0Norm_f_Omega} and
properties of solutions
of this problem \cite{Gol2}, we get the estimate
\begin{equation}\label{u_eps_omega_1}
\|u_\varepsilon\|_{\mathscr{W}_2^2(\omega_1)}\leq
c_5
\big(
|\kappa_\varepsilon|+
\varepsilon\|f(\varepsilon\,\cdot)\|_{\mathscr{L}_2(\Omega)}
\big)
\leq c_6\|f\|.
\end{equation}

Next, we claim that $\alpha$ does not belong to the intersection of spectra of the following problems
\begin{equation*}
\begin{cases}
-f''+\alpha Qf=0\quad\mathrm{on}\quad\omega_2,\\
\phantom{-}f(a)=0,\quad\frac{\mathrm{d} f}{\mathrm{d}\omega_2}(a_2)=0,
\end{cases}
\qquad
\begin{cases}
-f''+\alpha Qf=0\quad\mathrm{on}\quad\omega_3,\\
\phantom{-}f(a)=0,\quad\frac{\mathrm{d} f}{\mathrm{d}\omega_3}(a_3)=0.
\end{cases}
\end{equation*}
Assume the contrary, i.e., let there exist nonzero solutions of the above problems.
From these eigenfunctions one can construct in a straightforward manner an eigenfunction of the problem \eqref{Resonant_Set}
vanishing on $\omega_1$, which is impossible in view of the equality $\phi_\alpha(a_1)=1$.
Without loss of generality we suppose that
$\alpha$ does not belong to the spectrum of the problem on $\omega_2$.
Therefore the nonhomogenous problem
\begin{equation*}
\begin{cases}
-u''+\alpha Qu=\varepsilon f(\varepsilon\,\cdot)\quad\mathrm{on}\quad\omega_2,\\
\phantom{-}u(a)=u_{\varepsilon,\omega_1}(a),\quad\frac{\mathrm{d} u}{\mathrm{d}\omega_2}(a_2)=
\frac{\mathrm{d} y}{\mathrm{d}\gamma_2^\varepsilon}(a_2^\varepsilon)
\end{cases}
\end{equation*}
admits a unique solution which coincides with $u_{\varepsilon,\omega_2}$.
Moreover,
\begin{equation}\label{u_eps_omega_2}
\|u_\varepsilon\|_{\mathscr{W}_2^2(\omega_2)}\leq
c_7
\Big(
|u_{\varepsilon,\omega_1}(a)|
+
\Big|\frac{\mathrm{d} y}{\mathrm{d}\gamma_2^\varepsilon}(a_2^\varepsilon)\Big|
+
\varepsilon\|f(\varepsilon\,\cdot)\|_{\mathscr{L}_2(\Omega)}
\Big)
\leq c_8\|f\|
\end{equation}
by the a priori estimate of $u_{\varepsilon,\omega_2}$, \eqref{0Norm_f_Omega} and \eqref{u_eps_omega_1}.

Next, the Cauchy problem
\begin{equation*}
\begin{cases}
-u''+\alpha Qu=\varepsilon f(\varepsilon\,\cdot)\quad\mathrm{on}\quad\omega_3,\\
\phantom{-}u(a)=u_{\varepsilon,\omega_1}(a),\quad\frac{\mathrm{d} u}{\mathrm{d}\omega_3}(a)=-\frac{\mathrm{d} u_\varepsilon}{\mathrm{d}\omega_1}(a)-\frac{\mathrm{d} u_\varepsilon}{\mathrm{d}\omega_2}(a),
\end{cases}
\end{equation*}
gives us $u_{\varepsilon,\omega_3}$.
From this, using  \eqref{u_eps_omega_1} and \eqref{u_eps_omega_2}, we find that
\begin{equation*}
\|u_\varepsilon\|_{\mathscr{W}_2^2(\omega_3)}\leq
c_9
\Big(
|u_{\varepsilon,\omega_1}(a)|
+
\Big|\frac{\mathrm{d} u_\varepsilon}{\mathrm{d}\omega_1}(a)\Big|
+
\Big|\frac{\mathrm{d} u_\varepsilon}{\mathrm{d}\omega_2}(a)\Big|
+
\varepsilon\|f(\varepsilon\,\cdot)\|_{\mathscr{L}_2(\Omega)}
\Big)
\leq c_{10}\|f\|.
\end{equation*}
Combining the above estimate with \eqref{u_eps_omega_1} and \eqref{u_eps_omega_2}, we arrive at the desired inequality.
\end{proof}

Recall that $\chi_\varepsilon$ is the characteristic function of $\Omega_\varepsilon$ and put
\[
v_\varepsilon:=(1-\chi_\varepsilon)y+
\chi_\varepsilon\big[y(a_1^\varepsilon)\phi_\alpha(\varepsilon^{-1}\cdot)+\varepsilon u_\varepsilon(\varepsilon^{-1}\cdot)\big],
\]
i.e., $v_\varepsilon=y$ on $\Gamma\setminus\Omega_\varepsilon$ and $v_\varepsilon=y(a_1^\varepsilon)\phi_\alpha(\varepsilon^{-1}\cdot)+\varepsilon u_\varepsilon(\varepsilon^{-1}\cdot)$ on $\Omega_\varepsilon$.
We now estimate the jumps of $v_\varepsilon$ and its first derivative at the points $a_n^\varepsilon$.
Direct calculations show that
\[
[v_\varepsilon]_{a_n^\varepsilon}=
y(a_n^\varepsilon)-\theta_ny(a_1^\varepsilon)-\varepsilon u_\varepsilon(a_n),\qquad
[v'_\varepsilon]_{a_n^\varepsilon}=
\frac{\mathrm{d} y}{\mathrm{d}\gamma_n^\varepsilon}(a_n^\varepsilon)-\frac{\mathrm{d} u_\varepsilon}{\mathrm{d}\omega_n}(a_n).
\]
Using Lemma~\ref{Lemma_Simple_Resonant}, \eqref{1Jump_1_Der},
the relations $y_{\gamma_n}(a)=\theta_n\,y_{\gamma_1}(a)$,
and the estimates
\begin{equation}\label{1Jump_Y}
|y(a_n^\varepsilon)- y_{\gamma_n}(a)|
\leq \int_a^{a_n^\varepsilon}|y'|\,\mathrm{d}\gamma_n\leq c_1 \varepsilon^\frac12\|y\|_
{\mathscr{W}_2^2(\Gamma)}\leq
c_2\varepsilon^\frac12\|f\|
\end{equation}
holding for $n=1,2,3$, we arrive at the bounds for the jumps
\begin{align}
\label{1W_Jamps}
\begin{aligned}
|[v_\varepsilon]_{a_n^\varepsilon}|\leq&\, |y(a_n^\varepsilon)- y_{\gamma_n}(a)|+
|\theta_n|\,|y(a_1^\varepsilon)- y_{\gamma_1}(a)|+\varepsilon\,\|u_\varepsilon\|_{\mathscr{W}_2^2(\Omega)}\leq c_3\varepsilon^\frac12\|f\|,\\
|[v'_\varepsilon]_{a_n^\varepsilon}|\leq&\,
\Big|\frac{\mathrm{d} y}{\mathrm{d}\gamma_n^\varepsilon}(a_n^\varepsilon)-\frac{\mathrm{d} y}{\mathrm{d}\gamma_n}(a)\Big|+
\Big|\frac{\mathrm{d} y}{\mathrm{d}\gamma_n}(a_n)-\frac{\mathrm{d} u_\varepsilon}{\mathrm{d}\omega_n}(a_n)\Big|
\leq c_4\varepsilon^\frac12\|f\|.
\end{aligned}
\end{align}

Introduce now the function $\eta_{1,n}^\varepsilon$, supported by $[a_n^\varepsilon,a_n^{\varepsilon+1}]$, that is smooth outside the point $a_n^\varepsilon$, linear on $[a_n^\varepsilon,a_n^{\varepsilon+\frac12}]$ with $\eta_{1,n}^\varepsilon(a_n^\varepsilon)=0$ and $\eta_{1,n}^\varepsilon(a_n^{\varepsilon+\frac12})=\frac12$.
Put
\[
w_\varepsilon(x):=\sum_{n=1}^3([v_\varepsilon]_{a_n^\varepsilon}\eta_{0,n}^\varepsilon(x)+
[v'_\varepsilon]_{a_n^\varepsilon}\eta_{1,n}^\varepsilon(x)).
\]
Inequalities \eqref{1W_Jamps} imply that
\[
\max\limits_{x\in\Omega_{\varepsilon+1}\setminus\Omega_\varepsilon}
|w_\varepsilon^{(n)}(x)|\leq
c_5\sqrt\varepsilon\,\|f\|
\]
for $n=0,1,2$.
By construction,
\[
\tilde{y}_\varepsilon:=v_\varepsilon-w_\varepsilon=
\begin{cases}
y-w_\varepsilon&\text{on}\quad \Gamma\setminus\Omega_\varepsilon,\\
y(a_1^\varepsilon)\phi_\alpha(\varepsilon^{-1}\cdot)+\varepsilon\, u_\varepsilon(\varepsilon^{-1}\cdot)&\text{on}\quad \Omega_\varepsilon
\end{cases}
\]
is a $\mathscr{C}^1(\Gamma)$-function and belongs to the domain of $\mathrm{H}_\varepsilon(\alpha,Q)$.

The rest of the proof of Theorem 1 (i.e., proof of \eqref{Ineq_Approx}) in the simple resonant cases is similar to that in the non-resonant case and is therefore omitted.

\subsection{Proof of Theorem 1 in the double resonant case}
Finally, we establish Theorem 1 in the case when $\alpha$ is a double eigenvalue of problem \eqref{Resonant_Set}.
Let $\varphi_\alpha$ and $\psi_\alpha$ be a pair of linearly independent eigenfunctions of problem \eqref{Resonant_Set} cor\-responding to the eigenvalue $\alpha$.
Let $u_\varepsilon$ be any solution of the following problem
\begin{equation*}
\begin{cases}
-u''+\alpha Qu=\varepsilon f(\varepsilon\,\cdot)\quad\mathrm{on}\quad\Omega,\qquad
u\in \mathcal{K}(\Omega;a),\\
\phantom{-}\frac{\mathrm{d} u}{\mathrm{d}\omega_1}(a_1)=\mu_\varepsilon
,\quad
\frac{\mathrm{d} u}{\mathrm{d}\omega_2}(a_2)=\nu_\varepsilon
,\quad
\frac{\mathrm{d} u}{\mathrm{d}\omega_3}(a_3)=\frac{\mathrm{d} y}{\mathrm{d}\gamma_3^\varepsilon}(a_3^\varepsilon),
\end{cases}
\end{equation*}
where
\begin{align}\label{2MuNu_eps}
\begin{aligned}
\mu_\varepsilon&=
\frac{\theta_1}{\theta_3}
\frac{\mathrm{d} y}{\mathrm{d}\gamma_3^\varepsilon}(a_3^\varepsilon)
+
\frac\varepsilon{\theta_3}
\int_\Omega \big(\psi_\alpha(a_2)\varphi_\alpha(t)
-\varphi_\alpha(a_2)\psi_\alpha(t)\big)f(\varepsilon t)\,\mathrm{d}\Omega,\\
\nu_\varepsilon&=
\frac{\theta_2}{\theta_3}
\frac{\mathrm{d} y}{\mathrm{d}\gamma_3^\varepsilon}(a_3^\varepsilon)
+
\frac\varepsilon{\theta_3}
\int_\Omega \big(\varphi_\alpha(a_1)\psi_\alpha(t)
-\psi_\alpha(a_1)\varphi_\alpha(t)\big)f(\varepsilon t)\,\mathrm{d}\Omega.
\end{aligned}
\end{align}
This problem admits solutions in view of the Fredholm alternative.
We fix $u_\varepsilon$ by the conditions $u_\varepsilon(a_1)=u_\varepsilon(a_2)=0$.
The following statement is similar to Lemma~\ref{Lemma_Simple_Resonant}.

\begin{lemma}\label{Lemma_Double_Resonant}
Suppose that $f\in \mathscr{L}_2(\Gamma)$ and $\varepsilon\in(0,1)$.
Then
\[
\Big|\mu_\varepsilon-\frac{\mathrm{d} y}{\mathrm{d}\gamma_1}(a)\Big|
\leq
C_1\varepsilon^\frac12\|f\|,
\qquad
\Big|\nu_\varepsilon-\frac{\mathrm{d} y}{\mathrm{d}\gamma_2}(a)\Big|
\leq
C_2\varepsilon^\frac12\|f\|,
\qquad
\|u_\varepsilon\|_{\mathscr{W}_2^2(\Omega)}\leq C_3\|f\|.
\]
\end{lemma}
\begin{proof}
In a similar manner as in the previous subsections, one concludes that
\[
\|y\|_{\mathscr{W}_2^2(\Gamma)}\leq c_1\|f\|,\qquad
\|y\|_{\mathscr{C}^1(\Gamma)}\leq c_2\|f\|.
\]
Subtracting the relation $\frac{\mathrm{d} y}{\mathrm{d}\gamma_1}(a)=
\frac{\theta_1}{\theta_3}\frac{\mathrm{d} y}{\mathrm{d}\gamma_3}(a)$
from \eqref{2MuNu_eps} yields
\begin{align*}
\Big|\frac{\mathrm{d} u_\varepsilon}{\mathrm{d}\omega_1}(a_1)-\frac{\mathrm{d} y}{\mathrm{d}\gamma_1}(a)\Big|\leq
\Big|\frac{\theta_1}{\theta_3}\Big|
\Big|\frac{\mathrm{d} y}{\mathrm{d}\gamma_3^\varepsilon}(a_3^\varepsilon)-\frac{\mathrm{d} y}{\mathrm{d}\gamma_3}(a)\Big|
+c_3\varepsilon\|f(\varepsilon\,\cdot)\|_{L_2(\Omega)}\leq
C_1\varepsilon^\frac12\|f\|.
\end{align*}
Similarly, with
$\frac{\mathrm{d} y}{\mathrm{d}\gamma_2}(a)=
\frac{\theta_2}{\theta_3}\frac{\mathrm{d} y}{\mathrm{d}\gamma_3}(a)$,
\[
\Big|\frac{\mathrm{d} u_\varepsilon}{\mathrm{d}\omega_2}(a_2)-\frac{\mathrm{d} y}{\mathrm{d}\gamma_2}(a)\Big|\leq
\Big|\frac{\theta_2}{\theta_3}\Big|
\Big|\frac{\mathrm{d} y}{\mathrm{d}\gamma_3^\varepsilon}(a_3^\varepsilon)-\frac{\mathrm{d} y}{\mathrm{d}\gamma_3}(a)\Big|
+c_4\varepsilon\|f(\varepsilon\,\cdot)\|_{L_2(\Omega)}\leq
C_2\varepsilon^\frac12\|f\|.
\]

By construction, the restrictions of $u_{\varepsilon}$ on $\omega_1$ and $\omega_2$ obey the Cauchy problems
\begin{equation*}
\begin{cases}
-u''+\alpha Qu=\varepsilon f(\varepsilon\,\cdot)\quad\mathrm{on}\quad\omega_1,\\
\phantom{-}u(a_1)=0,\quad\frac{\mathrm{d} u}{\mathrm{d}\omega_1}(a_1)=\mu_\varepsilon,
\end{cases}
\qquad
\begin{cases}
-u''+\alpha Qu=\varepsilon f(\varepsilon\,\cdot)\quad\mathrm{on}\quad\omega_2,\\
\phantom{-}u(a_2)=0,\quad\frac{\mathrm{d} u}{\mathrm{d}\omega_2}(a_2)=\nu_\varepsilon;
\end{cases}
\end{equation*}
then, since \eqref{0Norm_f_Omega} holds, we arrive at
\begin{align}\label{u_eps_omega_12}
\begin{aligned}
\|u_\varepsilon\|_{\mathscr{W}_2^2(\omega_1)}&\leq
c_5\big(
|\mu_\varepsilon|+
\varepsilon\|f(\varepsilon\,\cdot)\|_{\mathscr{L}_2(\Omega)}
\big)
\leq c_6\|f\|,\\
\|u_\varepsilon\|_{\mathscr{W}_2^2(\omega_2)}&\leq
c_7\big(
|\nu_\varepsilon|+
\varepsilon\|f(\varepsilon\,\cdot)\|_{\mathscr{L}_2(\Omega)}
\big)
\leq c_8\|f\|.
\end{aligned}
\end{align}
We see that the function $u_{\varepsilon,\omega_3}$ is a solution of the initial problem
\begin{equation*}
\begin{cases}
-u''+\alpha Qu=\varepsilon f(\varepsilon\,\cdot)\quad\mathrm{on}\quad\omega_3,\\
\phantom{-}u(a)=u_{\varepsilon,\omega_1}(a),\quad\frac{\mathrm{d} u}{\mathrm{d}\omega_3}(a)=-\frac{\mathrm{d} u_\varepsilon}{\mathrm{d}\omega_1}(a)-\frac{\mathrm{d} u_\varepsilon}{\mathrm{d}\omega_2}(a)
\end{cases}
\end{equation*}
and that, moreover,
\begin{equation*}
\|u_\varepsilon\|_{\mathscr{W}_2^2(\omega_3)}\leq
c_9
\Big(
|u_{\varepsilon,\omega_1}(a)|
+
\Big|\frac{\mathrm{d} u_\varepsilon}{\mathrm{d}\omega_1}(a)\Big|
+
\Big|\frac{\mathrm{d} u_\varepsilon}{\mathrm{d}\omega_2}(a)\Big|
+
\varepsilon\|f(\varepsilon\,\cdot)\|_{\mathscr{L}_2(\Omega)}
\Big)
\leq c_{10}\|f\|
\end{equation*}
in view of \eqref{0Norm_f_Omega} and \eqref{u_eps_omega_12}.
The proof is complete.
\end{proof}

We set
\begin{align*}
v_\varepsilon:=\,
&(1-\chi_\varepsilon)y+
\frac{\chi_\varepsilon}{\theta_3}
\big(
\big\{\psi_\alpha(a_1)y(a_2^\varepsilon)
-\psi_\alpha(a_2)y(a_1^\varepsilon)\big\}
\varphi_\alpha(\varepsilon^{-1}\cdot)
\\
&+
\big\{\varphi_\alpha(a_2)y(a_1^\varepsilon)
-\varphi_\alpha(a_1)y(a_2^\varepsilon)\big\}
\psi_\alpha(\varepsilon^{-1}\cdot)
+\theta_3\varepsilon u_\varepsilon(\varepsilon^{-1}\cdot)\big).
\end{align*}
Straightforward calculations show that
\begin{align*}
[v_\varepsilon]_{a_n^\varepsilon}=&\,
y(a_n^\varepsilon)
-\frac{y(a_1^\varepsilon)}{\theta_3}
\big(\psi_\alpha(a_n)\varphi_\alpha(a_2)
-\psi_\alpha(a_2)\varphi_\alpha(a_n)\big)\\
&-\frac{y(a_2^\varepsilon)}{\theta_3}
\big(\psi_\alpha(a_1)\varphi_\alpha(a_n)
-\psi_\alpha(a_n)\varphi_\alpha(a_1)\big)
-\varepsilon u_\varepsilon(a_n),
\\
[v'_\varepsilon]_{a_n^\varepsilon}=&\,
\frac{\mathrm{d} y}{\mathrm{d}\gamma_n^\varepsilon}(a_n^\varepsilon)-\frac{\mathrm{d} u_\varepsilon}{\mathrm{d}\omega_n}(a_n).
\end{align*}
In view of Lemma~\ref{Lemma_Double_Resonant}, \eqref{1Jump_1_Der},
\eqref{1Jump_Y}, and the relation
$
y_{\gamma_1}(a)=
-
\frac{\theta_1}{\theta_2}y_{\gamma_2}(a)
-
\frac{\theta_1}{\theta_3}y_{\gamma_3}(a)
$,
we get
\begin{align*}
|[v_\varepsilon]_{a_n^\varepsilon}|\leq
&\, |y(a_n^\varepsilon)- y_{\gamma_n}(a)|+
\Big|\frac{\theta_1}{\theta_3}\Big|
|y(a_1^\varepsilon)- y_{\gamma_1}(a)|
\\
&+\Big|\frac{\theta_2}{\theta_3}\Big|
|y(a_2^\varepsilon)- y_{\gamma_2}(a)|
+\varepsilon\,\|u_\varepsilon\|_{\mathscr{W}_2^2(\Omega)}
\leq c_3\varepsilon^\frac12\|f\|,\\
|[v'_\varepsilon]_{a_n^\varepsilon}|\leq&\,
\Big|\frac{\mathrm{d} y}{\mathrm{d}\gamma_n^\varepsilon}(a_n^\varepsilon)-\frac{\mathrm{d} y}{\mathrm{d}\gamma_n}(a)\Big|+
\Big|\frac{\mathrm{d} y}{\mathrm{d}\gamma_n}(a_n)-\frac{\mathrm{d} u_\varepsilon}{\mathrm{d}\omega_n}(a_n)\Big|
\leq c_4\varepsilon^\frac12\|f\|.
\end{align*}
Now we set
\[
\tilde{y}_\varepsilon:=
\begin{cases}
y-w_\varepsilon&\text{on}\quad \Gamma\setminus\Omega_\varepsilon,\\
\frac{\psi_\alpha(a_1)y(a_2^\varepsilon)
-\psi_\alpha(a_2)y(a_1^\varepsilon)}{\theta_3}
\varphi_\alpha(\varepsilon^{-1}\cdot)\\
+
\frac{\varphi_\alpha(a_2)y(a_1^\varepsilon)
-\varphi_\alpha(a_1)y(a_2^\varepsilon)}{\theta_3}
\psi_\alpha(\varepsilon^{-1}\cdot)
+\varepsilon u_\varepsilon(\varepsilon^{-1}\cdot)
&\text{on}\quad \Omega_\varepsilon
\end{cases}
\]
with
\[
w_\varepsilon(x):=\sum_{n=1}^3([v_\varepsilon]_{a_n^\varepsilon}\eta_{0,n}^\varepsilon(x)+
[v'_\varepsilon]_{a_n^\varepsilon}\eta_{1,n}^\varepsilon(x)).
\]
The rest of the proof runs as before.

\section*{Acknowledgments}
The author expresses his gratitude to Professors Yuriy Golovaty and Rostyslav Hryniv for stimulating and fruitful discussions.


\begin{thebibliography}{99}
\bibitem{AlbCacFin07}
Albeverio, S.,  Cacciapuoti, C., and Finco, D., ``Coupling in the singular limit of thin quantum waveguides,''
J. Math. Phys. {\bf 48}, 032103 (2007).

\bibitem{Alb2ed}
S. Albeverio,  F. Gesztesy,  R. H{\o}egh-Krohn, and H. Holden,
{\it Solvable models in quantum mechanics. With an appendix by Pavel Exner} (AMS Chelsea Publishing, New York, 2005).

\bibitem{AlbKur}
S. Albeverio, and P. Kurasov,
{\it
Singular perturbations of differential operators
and solvable Schr\"{o}dinger type operators}
(Cambridge University Press, Cambridge, 2000).


\bibitem{Anderson81}
Anderson, P.,
``New method for scaling theory of localization. II. Multichannel theory of a ``wire'' and possible extension to higher dimensionality,'' Phys. Rev. B {\bf 23},  4828--4836 (1981).

\bibitem{AvishaiLuck92}
Avishai, Y., and Luck, J., ``Quantum percolation and ballistic conductance on a lattice of wires,'' Phys. Rev. B {\bf 45}, 1074--1095 (1992).

\bibitem{CacExn07}
Cacciapuoti, C., and Exner, P., ``Nontrivial edge coupling from a Dirichlet network squeezing: the
case of a bent waveguide,''
J. Phys. A: Math. Theor. {\bf 40}, F511--F523 (2007).

\bibitem{ChalkerCoddington88}
Chalker, J., and Coddington, P.,
``Percolation, quantum tunnelling and the integer Hall effect,''
J. Phys. C: Solid State Phys. {\bf 21}, 2665 (1988).


\bibitem{Exn1}
Cheon, T., Exner, P., and \v{S}eba, P.,
``Wave function shredding by sparse quantum barriers,''
Phys. Lett. A. {\bf 277}, 1--6 (2000).

\bibitem{Exn2}
Cheon, T., Exner, P., and Turek, O.,
``Approximation of a general singular vertex coupling in quantum graphs,''
Ann. Phys. {\bf 325}, 548--578 (2010).

\bibitem{ChrArnZolErmGai}
Christiansen, P., Arnbak, H., Zolotaryuk, A., Ermakov, V.,
and Gaididei, Y.,
``On the existence of resonances in the transmission probability for interactions arising from derivatives of Dirac's delta function,'' J. Phys. A: Math. Gen. {\bf 36}, 7589--7600 (2003).


\bibitem{Exn3}
Duclos, P., Exner, P., and Turek, O.,
``On the spectrum of a bent chain graph,''
J. Phys. A: Math. Theor. {\bf41}, 415206 (2008).


\bibitem{Exn4}
Exner, P.,
``Contact interactions on graph superlattices,''
J. Phys. A: Math. Gen. {\bf 29}, 87--102 (1996).

\bibitem{ExnerPost09}
Exner, P., and Post, O.,
``Approximation of quantum graph vertex couplings by scaled Schr\"{o}dinger operators on thin branched manifolds,''
J. Phys. A: Math. Theor. {\bf 42}, 415305 (2009).

\bibitem{FlesiaJohnstonKunz87}
Flesia, C., Johnston, R., and Kunz, H., ``Strong localization of classical waves: a numerical study,'' Europhys. Lett. {\bf 3}, 497
(1987).

\bibitem{GolMan1}
Golovaty, Yu., and Man'ko, S.,
``Schr\"{o}dinger operator with $\delta'$-potential,''
Dopov. Nats. Akad. Nauk Ukr., Mat. Pryr. Tekh. Nauky. {\bf5}, 16--21 (2009).

\bibitem{GolMan2}
Golovaty, Yu., and Man'ko, S.,
``Solvable models for the Schr\"{o}dinger operators with $\delta'$-like potentials,''
Ukr. Mat. Visn. {\bf 6}, 173--206 (2009).

\bibitem{GolHry}
Golovaty, Yu., and Hryniv, R.,
``On norm resolvent convergence of Schr\"{o}dinger operators with $\delta'$-like potentials,''
J. Phys. A: Math. Theor. {\bf 15}, 155204 (2010).


\bibitem{Gol1}
Golovaty, Yu.,
``Schr\"{o}dinger operators with $(\alpha\delta'+\beta\delta)$-like potentials: norm resolvent convergence and solvable models,''
arXiv:1201.2610v2 [math.SP].

\bibitem{Gol2}
Golovaty, Yu.,
``1D Schr\"{o}dinger operators with short range
interactions: two-scale regularization of
distributional potentials,''
arXiv:1202.4711v2 [math.SP].


\bibitem{KostrykinSchrader99}
Kostrykin, V., and Schrader, R.,
``Kirchoff's rule for quantum wires,''
J. Phys. A: Math. Gen. {\bf 32}, 595--630 (1999).


\bibitem{KostrykinSchrader00}
Kostrykin, V., and Schrader, R.,
``Kirchoff's rule for quantum wires II. The inverse problem with possible applications to quantum computers,''
Fortschrt. Phys. {\bf 48}, 703--716 (2000).

\bibitem{KostrykinSchrader03}
Kostrykin, V., and Schrader, R.,
``Quantum wires with magnetic fluxes,''
Commun. Math. Phys. {\bf 237}, 161--179 (2003).


\bibitem{KottosSmilansky97}
Kottos, T., and Smilansky, U.,
``Quantum chaos on graphs,''
Phys. Rev. Lett. {\bf 79}, 4794--4797 (1997).

\bibitem{KowalSivanEntinWohlmanImry90}
Kowal, D., Sivan, U., Entin-Wohlman, O., and Imry, Y.,
``Transmission through multiply-connected wire systems,''
Phys. Rev. B {\bf 42}, 9009--9018 (1990).

\bibitem{Kuchment02}
Kuchment, P.,
``Graph models for waves in thin structures,''
Waves Random Media {\bf 12}, R1--R24 (2002).

\bibitem{Kuchment04}
Kuchment, P.,
``Quantum graphs I. Some basic structures,''
Waves Random Media {\bf 14}, S107--S128 (2004).

\bibitem{KuchmentKunyansky02}
Kuchment, P., and Kunyansky, L.,
``Differential operators on graphs and photonic crystals,''
Adv. Comput. Math. {\bf 16}, 263--290 (2002).

\bibitem{Man1}
Man'ko, S.,
``On Schr\"{o}dinger and Sturm-Liouville operators with $\delta'$-potentials,''
Visn. L'viv. Univ., Ser. Mekh.-Mat.
{\bf 71}, 142--155 (2009).

\bibitem{Man2}
Man'ko, S.,
``On Schr\"{o}dinger operator with singular potential for geometric graphs,''
Scientific Herald of Yuriy Fedkovych Chernivtsi National University. Series: mathematics:
collection of scientific articles. {\bf1}, 61--71 (2011).

\bibitem{Man3}
Man'ko, S.,
``On $\delta'$-like potential scattering on star graphs,''
J. Phys. A: Math. Theor. {\bf43}, 445304 (2010).


\bibitem{NaimarkSolomyak00}
Naimark, K., and Solomyak, M.,
``Eigenvalue estimates for the weighted Laplacian on metric trees,'' Proc. London Math. Soc. {\bf 80}, 690--724 (2000).

\bibitem{Pauling36}
Pauling, L.,
``The diamagnetic anisotropy of aromatic molecules,''
J. Chem. Phys. {\bf 4}, 673--677 (1936).


\bibitem{Platt49}
Platt, J.,
``Classification of spectra of cata-condensed hydrocarbons,''
J. Chem. Phys. {\bf 17}, 484--495 (1949).


\bibitem{RichardsonBalazs72}
Richardson, M., and Balazs, N.,
``On the network model of molecules and solids,''
Ann. Phys. {\bf 73}, 308--325 (1972).

\bibitem{Seb}
\v{S}eba, P.,
``Some remarks on the $\delta'$-interaction in one dimension,''
Rep. Math. Phys. {\bf 1}, 111--120 (1986).

\bibitem{Solomyak03}
Solomyak, M.,
``On approximation of functions from Sobolev spaces on metric graphs,''
J. Approx. Theory {\bf 121}, 199--219 (2003).

\bibitem{Solomyak04}
Solomyak, M.,
``On the spectrum of the Laplacian on regular metric trees,''
Waves Random Media {\bf 14}, 155--171 (2004).


\bibitem{Zol1}
Zolotaryuk, A., Christiansen, P., and Iermakova, S.,
``Scattering properties of point dipole interactions,''
J. Phys. A: Math. Gen. {\bf39}, 9329--9338 (2006).

\bibitem{Zol2}
Zolotaryuk, A., Christiansen, P., and Iermakova, S.,
``Resonant tunneling through short-range singular
potentials,''
J. Phys. A: Math. Theor. {\bf 40}, 5443--5457  (2007).

\bibitem{Zol3}
Zolotaryuk, A.,
``Two-parametric resonant tunneling across the $\delta'(x)$-potential,''
Adv. Sci. Lett. {\bf 1}, 187--191 (2008).

\bibitem{Zol4}
Zolotaryuk, A.,
``Point interactions of the dipole type defned through a three-parametric power regularization,''
J. Phys. A: Math. Theor. {\bf 43}, 105302 (2010).

\bibitem{Zol5}
Zolotaryuk, A.,
``Boundary conditions for the states with resonant tunnelling across the $\delta'$-potential,''
Phys. Let. A {\bf 374}, 1636--1641 (2010).

\end{thebibliography}
\end{document}